\documentclass[fleqn,11pt]{article}
\usepackage{amsmath}
\usepackage{latexsym}
\usepackage{rotate,epsfig}
\usepackage{lscape}
\newtheorem{theorem}{\bf Theorem}[section]
\newtheorem{proposition}[theorem]{\bf Proposition}

    {\par \noindent {\bf Example }}%
    {\par \indent}
    {\par \noindent {\bf Remark }}%
    {\par \indent}
\newenvironment{proof}%
    {\par \noindent {\bf Proof}}%
    {\par \indent}
\begin{document}
\title{{\bf Goodness-of-fit tests for weibull populations on the basis of records}}
\author{Mahdi Doostparast\footnote{
\textit{E-mail addresses:} doostparast@math.um.ac.ir}
\\
{\small {\it Department of Statistics, School of
Mathematical Sciences,}}\vspace{-0.2cm}\\ {\small {\it Ferdowsi University of Mashhad, P. O. Box 91775-1159,  Mashhad,\ Iran }}\\
}
\date{}
\maketitle

\begin{abstract}
Record is used to reduce the time and cost of running experiments (Doostparast and Balakrishnan, 2010).
It is important to check the adequacy
of models upon which inferences or actions are based (Lawless, 2003, Chapter 10, p. 465). In the area of goodness of fit based on record data, there are a few works. Smith (1988) proposed a form of residual for testing some parametric models. But in most cases, the variation
inherent in graphical summaries is substantial, even when the
data are generated by assumed model, and the eye can not always
determine whether features in a plot are within the bounds of
natural random variation. Consequently, formal hypothesis tests
are an important part of model checking (Lawless, 2003).\\
In this paper, Kolmogorov-Smirnov and  Cramer-von Mises type goodness of fit tests for record data are proposed. Also a new weighted goodness of fit test is suggested. A Monte-Carlo simulation study is conducted to derive the percentiles of the statistics proposed. Finally, some real data sets are given to investigate results obtained.
\end{abstract}


\section{Introduction}\label{intro}
In reliability, we are concerned primarily with test data in
which lifetimes of items that fail during the course of the test
are recorded or with variables related in some way to item
lifetimes. If the actual lifetime of every item in the sample is
recorded, the data are  {\em complete} data. To obtain complete
data, it is necessary to continue the experiment until the last
item on test or in service has failed. In cases where even a few
items in the sample may have very long lifetimes, experiment can
go on for a very long period of time and, in fact, well beyond
the point at which the results may no longer be of any interest
or use.  In such situations, it may be desirable to terminate the
study prior to failure of all items under test. When observation
is discontinued prior to all items having failed, we obtain the
so-called {\em censored} data. There are a variety of forms of
censored data that arise in practice; See, for example,
Balakrishnan and Cohen (1991) and Cohen (1991).

A form of censored data that is often encountered in applications is the  so-called {\em record} data.
As pointed out by Gulati and Padgett (1995), often, in industrial testing,
meteorological data, and some other situations, measurements may be made
sequentially and only values smaller (or larger) than all previous ones are recorded. Such data may be represented by
$(\textbf{r,k}):=(r_1,k_1,r_2,k_2,\cdots,r_m,k_m)$, where $r_i$ is
the $i$-th record value meaning new minimum (or maximum) and $k_i$ is the number of
trials following the observation of $r_i$ that are needed to obtain a new
record value (or to exhaust the available observation). There are two sampling schemes for generating such a record-breaking data:
\begin{itemize}
\item ({\em Inverse sampling scheme})
Items are presented sequentially and sampling is terminated when the $m$-th minimum is observed.
In this case, the total number of items sampled is a random number, and $K_m$ is defined to be one for convenience;
\item ({\em Random sampling scheme}) A random sample $Y_1,\cdots,Y_n$ is examined sequentially
and successive minimum values are recorded.  In this setting, we
have $N^{(n)}$, the number of records obtained, to be random and,
given a value of $m$, we have in this case $\sum_{i=1}^m K_i=n$.
\end{itemize}

A random variable $X$ is said to
have an exponential distribution, denoted by $X\sim
Exp(\sigma)$, if its cumulative distribution function
(cdf) is
\begin{equation}\label{cdf:exp}
F(x;\sigma)=1-\exp\left\{-\left(\frac{x}{\sigma}\right)\right\},\
\ \ x\geq 0,\ \ \ \sigma>0,
\end{equation}
and the probability density function (pdf) is
\begin{equation}\label{pdf:exp}
f(x;\sigma)=\frac{1}{\sigma}\exp\left\{-\left(\frac{x}{\sigma}\right)\right\},\
\ \ x\geq 0,\ \ \ \sigma>0.
\end{equation}
The exponential distribution is commonly used in many applied problems. Such a exponential distribution is a natural model while studying a variable that can take on only positive values such as lifetime of units.
In some situations, the Weibull distribution is more suitable
than the exponential distributions
(Nelson, 1985). The Weibull cdf, denoted by
$W(\alpha,\sigma)$, is
\begin{equation}\label{cdf:weibull}
F(x;\alpha,\sigma)=1-\exp\left\{-\left(\frac{x}{\sigma}\right)^\alpha\right\},\ \ \ \alpha>0,\ \ \sigma>0,
\end{equation}
and hence with pdf
\begin{equation}\label{pdf:weibull}
f(x;\alpha,\sigma)=\frac{\alpha x^{\alpha-1}}{\sigma^\alpha}\exp\left\{-\left(\frac{x}{\sigma}\right)^\alpha\right\},\ \ \ \alpha>0,\ \ \sigma>0.
\end{equation}
The scale parameter $\sigma$ is called the characteristic life
because it is always 63.2-th percentile. It determines the spread
and has the same units as failure times, for example hours,
months, cycles, and so forth. Parameter $\alpha$ is a unitless
pure number and determines the shape of the distribution. For
$\alpha=1$, the Weibull distribution is the exponential
distribution. The Weibull distribution
appears very frequent in practical problems when we observe data
representing minimal values. For example, the life of a capacitor
is determined by shortest-lived portion of dielectric. For many parent populations with limited left tail, the
limit of the minimum of independent samples converges to a
Weibull distribution (Lawless, 2003).
Researchers often like to make parametric assumptions on the underlying distribution.
With this in mind, estimation of the mean of an exponential distribution based on record data has been treated by Samaniego and Whitaker (1986) and Doostparast (2009). Hoinkes and Padgett (1994) obtained the ML estimators from record-breaking data in this model.

As pointed out by Lawless (2003, Chapter 10, p. 465), it is important to check the adequacy
of models upon which inferences or actions are based. In the area of goodness of fit based on record data, there is a
lack of published literature. But, there are a few works in this direction. However, informal
methods of model checking emphasize graphical procedures such as
probability and residual plots, Smith (1988) proposed a form of residual for testing some parametric models. But in most cases, the variation
inherent in graphical summaries is substantial, even when the
data are generated by assumed model, and the eye can not always
determine whether features in a plot are within the bounds of
natural random variation. Consequently, formal hypothesis tests
are an important part of model checking.  \\

Motivated by this, the aim of this paper is to provide some methods for model checking on
the basis of records. Specifically, suppose that the record data $\{R_1,K_1,\cdots,R_m,K_m\}$ are
coming from a population with parent cdf $F(.)$. We consider testing
\begin{equation}\label{goodness:weibull:h0}
H_0:F(x)=1-\exp\left\{-\left(\frac{x}{\sigma}\right)^\alpha\right\},\ \ \ \forall\ \ x\in (0,+\infty).
\end{equation}
where $\alpha$ and $\sigma$ may be unknown positive constants. In other
word, is the weibull model adequate to fit the data? 
Therefore, the rest of this article is organized as follows. Since weibull model has a wide variety application, in Section 2, maximum likelihood estimate (MLE) of the unknown parameters in Weibull model are obtained. In Section 3, explicit expression for Kolmogorov-Smirnov (K-S) and Cramer-Misses (C-M) goodness of fit tests is derived and we proposed a new modified goodness of fit test which is more suitable than the K-S and C-M statistics for records. Critical values of these statistics are obtained by a simulation study. In Section 4, Exponential model is considered and goodness of fit test for exponential model against the alternative weibull model is obtained. Finally, some numerical examples are given to investigate results obtained.\\

\section{Fitting a Weibull model}
It can be shown that, the likelihood function for the two sampling
schemes is given by
\begin{equation}\label{likelihood}
L(\theta)\equiv \prod_{i=1}^m
f(r_i)\left\{1-F(r_i)\right\}^{k_i-1},\ \ \ 0<r_m<\cdots<r_2<r_1.
\end{equation}
Let us assume that the sequence $\{R_1,K_1,\cdots,R_m,K_m\}$ are
coming from $W(\alpha,\sigma)$-model. The corresponding likelihood
function under either random or inversely sampling is obtained as
\begin{equation}\label{like:weibull}
L(\theta)\equiv\frac{\alpha^m}{\sigma^{m\alpha}}\left\{\prod_{i=1}^m r_i
\right\}^{\alpha-1}\exp\left\{-\frac{1}{\sigma^\alpha}\sum_{i=1}^m
k_i r_i^\alpha\right\}.
\end{equation}
After taking logarithm, we have
\begin{equation}\label{log:like:weibull}
l(\theta)\equiv m\log
(\alpha)-m\alpha\log(\sigma)+(\alpha-1)\sum_{i=1}^m
\log(r_i)-\frac{1}{\sigma^\alpha}\sum_{i=1}^m k_i r_i^\alpha.
\end{equation}
Through this paper "log" denotes natural logarithm. One can
easily show that, the maximum of \eqref{log:like:weibull} for
$m\geq 2$, by taking derivatives, is obtained from solving the
equations
\begin{equation}\label{weibull:log:likelihood:derivative}
\sigma=\left\{\frac{1}{m}\sum_{i=1}^m k_i
r_i^\alpha\right\}^{1/\alpha},
\end{equation}
and
\begin{equation}\label{weibull:log:likelihood:derivative:zero:alpha}
h(\alpha)=\frac{1}{m}\sum_{i=1}^m \ln r_i,
\end{equation}
where
$$h(\alpha)=\frac{\sum_{i=1}^m k_i r_i^\alpha\ln r_i}{\sum_{i=1}^m k_i
r_i^\alpha}-\frac{1}{\alpha}.$$ The equation
\eqref{weibull:log:likelihood:derivative:zero:alpha} cannot be
solved explicitly and hence the MLEs must be found by numerical
methods. These equations is similar with equations (6.2) and
(6.3) of Lehmann and Casella (1998, Ch. 6, p. 468). Hence, one can show that these equations have a unique solution. 

\section{GOF for weibull model}\label{gof:weibull}
GOF tests can be based on the approaches of comparison of
parametric estimates with nonparametric counterparts. Two well
known examples are the Kolmogorov-Smirnov (K-S) and the Cramer-von Mises (C-M) statistics defined by
\begin{equation}\label{KS}
\hat{D}_n=\sup_{-\infty<x<+\infty} |\hat{F}(x)-F_0(x)|,
\end{equation}
and
\begin{equation}\label{CM}
\hat{W}_n^2=n\int_{-\infty}^{+\infty}\left\{\hat{F}(x)-F_0(x)\right\}^2 d F_0(x),
\end{equation}
respectively, where $F_0(x)$ is the hypothesized model while $\hat{F}(x)$ is
the corresponding nonparametric maximum likelihood estimation (NPMLE). On the basis of
record data, arising from a random sample with size $n$,
Samaniego and Whitaker (1988) obtained NPMLE of survival function
$\bar{F}(x):=1-F(x)$ as
\begin{equation}\label{npmle:record}
\hat{\bar{F}}(x)=\prod_{i:r_{(i)}\leq x}\frac{\sum_{j=i}^m
k_{(j)}-1}{\sum_{j=i}^m k_{(j)}},
\end{equation}
where $r_{(0)}\equiv 0$ and $r_{(1)}<r_{(2)}<\cdots<r_{(m)}$ are
the observed record values, ordered from smallest to largest and
$\{k_{(i)}\}$ are the induced order statistics corresponding to
the ordered record values $\{r_{(i)}\}$ or $k_{(i)}=k_{m-i+1}$, $i=1,2,\cdots, m$.
As mentioned by Samaniego and Whitaker (1988), NPMLE in
\eqref{npmle:record} will perform poorly when estimating the
right tail of the actual distribution, thus we suggest a new GOF
statistic as follows
\begin{equation}\label{DS}
DS_n=n\int_0^{+\infty}\left(\hat{\bar{F}}(x)-\bar{F}_0(x)\right)^2\frac{1}{F_0(x)}d
F_0(x).
\end{equation}
The basic idea for $DS_n$ is similar with
Anderson-Darling statistic and is to measure the distance between
$\hat{F}(x)$ and $F_0(x)$ in left tail region of
$F_n(x)$ better than C-M statistic in \eqref{CM}. One may notice
that, on the basis of record data, the statistics $D_n$, $W_n^2$
and $DS_n$ are modified so that the supreme and integral are over
the range $y\leq r_1$. Sufficiently large values of $D_n$,
$W_n^2$ or $DS_n$ provide evidence against the hypothesized model.
To calculate the test statistics, the following Proposition is
helpful.
\begin{proposition}
Let $R_1,K_1,\cdots,R_m,K_m$ be record data arising from a random
sample with size $n$. Then the statistics $D_n$, $W_n^2$ and
$DS_n$ are simplified as
\begin{eqnarray}
D_n&=&\max_{1\leq i\leq
n}\left\{\max\left\{\hat{\Phi}_1\cdots\hat{\Phi}_{i-1}-\hat{\bar{F}}_0(r_{(i)}),
\hat{\bar{F}}_0(r_{(i)})-\hat{\Phi}_1\cdots\hat{\Phi}_i\right\}\right\},\label{KS:simple}\\
W_n^2&=&\frac{n}{3}\sum_{i=1}^{m+1}
\left\{\left[\hat{\Phi}_1\cdots\hat{\Phi}_{i-1}-\hat{\bar{F}}_0(r_{(i)})\right]^3-\left[\hat{\Phi}_1\cdots\hat{\Phi}_{i-1}-\hat{\bar{F}}_0(r_{(i-1)})\right]^3\right\},\label{CM:simple}\\
\mbox{and}&&\nonumber\\
 DS_n&=&n\left[\sum_{i=1}^{m+1}
\left\{\left[\hat{\Phi}_1\cdots\hat{\Phi}_{i-1}-1\right]^2\ln \hat{F}_0(r_{(i)})-\left[\hat{\Phi}_1\cdots\hat{\Phi}_{i-1}-1\right]^2\ln \hat{F}_0(r_{(i-1)})\right\},\right.\nonumber\\
&&\left.+2\sum_{i=1}^{m+1}
\left\{\left[\hat{\Phi}_1\cdots\hat{\Phi}_{i-1}-1\right]
\hat{F}_0(r_{(i)})-\left[\hat{\Phi}_1\cdots\hat{\Phi}_{i-1}-1\right]
\hat{F}_0(r_{(i-1)})\right\}+\frac{1}{2}\right],\label{DS:simple}
\end{eqnarray}
respectively, where $r_{(0)}\equiv 0$, $x_{m+1}\equiv +\infty$
and for $i=1$, $\hat{\Phi}_1\cdots\hat{\Phi}_{i-1}=1$ and
\[\hat{\Phi}_i=\frac{\sum_{j=i}^m
k_{(j)}-1}{\sum_{j=i}^m k_{(j)}},\ \ \ 1\leq i\leq m.\]
\end{proposition}
\begin{proof}
Proof of \eqref{KS:simple} is clear. For \eqref{CM:simple}, we
have
\begin{eqnarray*}
W_n^2&=&n\int_{0}^{+\infty}\left\{\hat{F}_n(y)-F_0(y)\right\}^2
dF_0(y)\\
&=&n\int_{0}^{+\infty}\left\{\hat{\bar{F}}_n(y)-\bar{F}_0(y)\right\}^2
dF_0(y)\\
&=&
n\sum_{i=1}^{m+1}\int_{r_{(i-1)}}^{r_{(i)}}\left\{\hat{\bar{F}}_n(y)-\bar{F}_0(y)\right\}^2
dF_0(y)\\
&=&
n\sum_{i=1}^{m+1}\int_{r_{(i-1)}}^{r_{(i)}}\left\{\hat{\Phi}_1\cdots\hat{\Phi}_{i-1}-\bar{F}_0(y)\right\}^2
dF_0(y)\\
&=&
n\sum_{i=1}^{m+1}\int_{F_0(r_{(i-1)})}^{F_0(r_{(i)})}\left\{\hat{\Phi}_1\cdots\hat{\Phi}_{i-1}-1+u\right\}^2
du\\
&=&
\frac{n}{3}\sum_{i=1}^{m+1}\left[\left\{\hat{\Phi}_1\cdots\hat{\Phi}_{i-1}-1+\hat{F}_0(r_{(i)})\right\}^3-
\left\{\hat{\Phi}_1\cdots\hat{\Phi}_{i-1}-1+\hat{F}_0(r_{(i-1)})\right\}^3\right].
\end{eqnarray*}
Similarly, one can show \eqref{DS:simple} and desired result
follows. \hfill{$\Box$}
\end{proof}

\begin{proposition}\label{indep:dist}
Assuming $H_0:F_0(y)=1-\exp\left\{-(x/\sigma)^\alpha\right\}$ is
true. Conditionally on $\{N^{(n)}\geq 2\}$, the distribution of
$D_n$, $W_n^2$ and $DS_n$, on the basis of record data do not
depend on $F_0(y)$.
\end{proposition}
\begin{proof}
Suppose $\{N^{(n)}\geq 2\}$. Let $R_i'\stackrel{D}{\equiv}
(R_i/\sigma)^\alpha$. Thus, $R_1',K_1,\cdots,R_m',K_m$ are coming
from a random sample with common distribution function $W(1,1)$.
The ML estimates on the basis of $R_1',K_1,\cdots,R_m',K_m$,
denoted by $\hat{\alpha}'$ and $\hat{\sigma}'$, are obtained by
solving \eqref{weibull:log:likelihood:derivative} and
\eqref{weibull:log:likelihood:derivative:zero:alpha} replacing
$r_i$ with $r_i'$. One can easily verify that
$\hat{\alpha}=\alpha\hat{\alpha}'$. This implies that
\begin{eqnarray*}
\hat{\sigma}&=&\left\{\frac{1}{m}\sum_{i=1}^m K_i
R_i^{\hat{\alpha}} \right\}^{\frac{1}{\hat{\alpha}}}\\
&=&\left\{\frac{1}{m}\sum_{i=1}^m K_i
(\sigma R_i'^{1/\alpha})^{\hat{\alpha}} \right\}^{\frac{1}{\hat{\alpha}}}\\
&=&\sigma\left\{\frac{1}{m}\sum_{i=1}^m K_i
(R_i')^{\hat{\alpha}/\alpha} \right\}^{\frac{1}{\hat{\alpha}}}\\
&=&\sigma\left\{\frac{1}{m}\sum_{i=1}^m K_i (R_i')^{\hat{\alpha}'}
\right\}^{\frac{1}{\hat{\alpha}'\alpha}}\\
&=&\sigma\left\{\hat{\sigma}'\right\}^{1/\alpha}.
\end{eqnarray*}
Hence, the estimate of weibull distribution function is obtained
as
\begin{eqnarray}
\hat{F}_0(x;\alpha,\sigma)&=&F_0(x;\hat{\alpha},\hat{\sigma})\nonumber\\
&=&1-\exp\left\{-\left(\frac{x}{\hat{\sigma}}\right)^{\hat{\alpha}}\right\}\nonumber\\
&=&1-\exp\left\{-\left(\frac{x}{\sigma\left\{\hat{\sigma}'\right\}^{1/\alpha}}\right)^{\alpha\hat{\alpha}'}\right\}\nonumber\\
&=&1-\exp\left\{-\left(\left[\frac{x}{\sigma}\right]^\alpha\frac{1}{\hat{\sigma}'}\right)^{\hat{\alpha}'}\right\}\nonumber\\
&=&1-\exp\left\{-\left(\frac{y}{\hat{\sigma}'}\right)^{\hat{\alpha}'}\right\}\nonumber\\
&=&\hat{F}^\star(y;\alpha,\sigma).\label{equ:cdf}
\end{eqnarray}
Similarly to Liao and Shimokawa (1999), this equation indicates
that $\hat{F}_0(x;\alpha,\sigma)$ is independent of the "true
values" of the parameters $\alpha$ and $\sigma$. This implies that $D_n$, $W_n^2$ and $DS_n$ is not depend on the "true value" of $\alpha$ and $\sigma$ when the parameters are estimated by the MLEs. The desired result follows. \hfill{$\Box$}
\end{proof}

Proposition \ref{indep:dist} clarifies that the distribution of
$D_n$, $W_n^2$ and $DS_n$, on the basis of record data, can be calculated via simulation without loss of generality by using a weibull distribution with $\alpha=\sigma=1$. Let $D_{n,\gamma}$, $W^2_{n,\gamma}$ and $DS_{n,\gamma}$ denotes the $\gamma$-th quantile of the distribution of $D_n$, $W_n^2$ and $DS_n$, on the basis of record data, respectively. These tests rejects the null hypothesis $H_0:F(x)=1-\exp\left\{-(x/\sigma)^\alpha\right\}$ of size $\gamma$, if the used GOF statistic exceeds its corresponding $(1-\gamma)$-th quantile. Table \ref{tab:gof:weibull} presents simulated critical values provided by a Monte-Carlo method. For this task, MC simulation provides the total sets of $M=100,000$ record samples and the values of $D_n$, $W_n^2$ and $DS_n$ are calculated and increasingly ordered. Then the critical values of $D_n$, $W_n^2$ and $DS_n$ for some significant level were calculated.
\begin{table}
\caption{Percentiles of $D_n$, $W_n^2$ and $DS_n$ for GOF of weibull model.} \label{tab:gof:weibull}
{\tiny
\begin{tabular}{cc|ccccccccc}\hline
$n$&&&&&$\gamma$&&\\\hline
&&0.01&0.025&0.05&0.1&0.5&0.90&0.95&0.975&0.99\\
&$D_n$&0.1758&0.2008&0.2253&0.2584&0.4445&0.8093&0.8627&0.8846&0.8976\\
5&$W_n^2$&0.0108&0.0166&0.0252&0.0414&0.2749&0.8842&1.0706&1.1545&1.2063\\
&$DS_n$&0.3819&0.4546&0.4963&0.5499&1.0480&2.6889&3.8012&4.4511&4.9176\\
&&&&&&&&&\\
&$D_n$&0.1508&0.1646&0.1877&0.2372&0.5296&0.8854&0.9170&0.9361&0.9494\\
10&$W_n^2$&0.0786&0.1354&0.2124&0.3524&0.9664&2.3140&2.5707&2.7348&2.8530\\
&$DS_n$&0.9747&1.0608&1.1891&1.4090&2.4504&8.9519&11.5462&13.7890&15.8577\\
&&&&&&&&&\\
&$D_n$&0.0858&0.1047&0.1394&0.2109&0.6430&0.9322&0.9502&0.9611&0.9704\\
20&$W_n^2$&0.7155&0.9951&1.1811&1.3369&3.1507&5.4022&5.7196&5.9185&6.0919\\
&$DS_n$&2.7572&3.1844&3.5657&4.0448&8.1438&26.5699&31.9874&36.4863&41.5800\\
&&&&&&&&&\\
&$D_n$&0.0451&0.0676&0.1067&0.1863&0.7763&0.9663&0.9743&0.9797&0.9846\\
50&$W_n^2$&3.4983&3.6557&3.7911&3.9573&11.4929&15.0385&15.4166&15.6718&15.9096\\
&$DS_n$&10.9018&11.6437&12.3496&13.4346&38.7320&98.0011&110.7805&121.8708&135.2813\\

\hline
\end{tabular}
}
\centering
\end{table}
\section{GOF for exponential model}\label{gof:expo}
As mentioned earlier, the model $W(\alpha,\sigma)$ reduces to
$Exp(\sigma)$ model when $\alpha=1$. Therefore, in this case,
testing the hypothesis $H_0:X\sim Exp(\sigma)$ against the
alternative $H_1:X\sim W(\alpha,\sigma)$ is equivalent to testing
$H_0:\alpha=1$ against the alternative $H_1:\alpha\neq 1$. We
could not find a UMP test of size $\gamma$ ($0<\gamma<1$) for
this hypothesis testing problem. We leave it as an open problem.
Therefore, we used the generalized likelihood ratio (GLR)
procedure in order to test these hypotheses. From
\eqref{cdf:weibull}, \eqref{pdf:weibull} and \eqref{like:weibull},
likelihood ratio statistic for testing $H_0:\alpha=1$ against the
alternative $H_1:\alpha\neq 1$ is given by
\begin{eqnarray}
\Lambda&=&\frac{\sup_{H_0} L}{\sup_{H_1}
L}\nonumber\\
&=&\frac{1}{\hat{\sigma}_{0}^m}\exp\left\{-\frac{1}{\hat{\sigma}_{0}}\sum_{i=1}^m
k_i r_i
\right\}\left(\frac{\hat{\alpha}^m}{\hat{\sigma}^{m\hat{\alpha}}}\left\{\prod_{i=1}^m
r_i
\right\}^{\hat{\alpha}-1}\exp\left\{-\frac{1}{\hat{\sigma}^\alpha}\sum_{i=1}^m
k_i
r_i^{\hat{\alpha}}\right\}\right)^{-1}\nonumber\\
&=&\left(\frac{m}{\sum_{i=1}^{m}k_i
r_i}\right)^m\exp\{-m\}\left(\frac{\hat{\alpha}^m}{\hat{\sigma}^{m\hat{\alpha}}}\left\{\prod_{i=1}^m
r_i \right\}^{\hat{\alpha}-1}\exp\left\{-m\right\}\right)^{-1}\nonumber\\
&=&\left(\frac{m}{\sum_{i=1}^{m}k_i
r_i}\right)^m\left(\frac{\hat{\alpha}^m}{\hat{\sigma}^{m\hat{\alpha}}}\left\{\prod_{i=1}^m
r_i \right\}^{\hat{\alpha}-1}\right)^{-1}\nonumber\\
&=&\left(\frac{m}{\sum_{i=1}^{m}k_i
r_i}\right)^m\hat{\sigma}^{m\hat{\alpha}}\left(\hat{\alpha}^m\left\{\prod_{i=1}^m
r_i \right\}^{\hat{\alpha}-1}\right)^{-1}\nonumber\\
&=&\left(\frac{\sum_{i=1}^{m}k_i
r_i^{\hat{\alpha}}}{\sum_{i=1}^{m}k_i
r_i}\right)^m\left(\hat{\alpha}^m\left\{\prod_{i=1}^m r_i
\right\}^{\hat{\alpha}-1}\right)^{-1}
\end{eqnarray}
where $\hat{\alpha}$ is obtained by solving equation
\eqref{weibull:log:likelihood:derivative:zero:alpha} and is  the
maximum likelihood estimation of $\alpha$ under $H_1$ while
$\hat{\sigma}_0$ is the ML estimate of $\sigma$ under $H_0$ and
is given by $\sum_{i=1}^m K_i R_i/n$.
\begin{proposition} When $\sigma$ is unknown,
critical region of the GLR test of level $\gamma$ for testing
$H_0:\alpha=1$ against the alternative $H_1:\alpha\neq 1$ is given
by
\begin{equation}\label{weibull:goodness:expansion:critical:glr}
C=\left\{({\bf r,k}):  \left(\frac{\sum_{i=1}^{m}k_i
r_i^{\hat{\alpha}}}{\sum_{i=1}^{m}k_i
r_i}\right)^m\left(\hat{\alpha}^m\left\{\prod_{i=1}^m r_i
\right\}^{\hat{\alpha}-1}\right)^{-1} < C^\star \right\},
\end{equation}
$\hat{\alpha}$ is the maximum likelihood estimation of $\alpha$
under $H_1$ and $C^\star$ is obtained from the size restriction
\begin{equation}\label{cstar:bothunknown:theta:glr}
\gamma=P_{\alpha=1}\left(\left(\frac{\sum_{i=1}^{m}k_i
r_i^{\hat{\alpha}}}{\sum_{i=1}^{m}k_i
r_i}\right)^m\left(\hat{\alpha}^m\left\{\prod_{i=1}^m r_i
\right\}^{\hat{\alpha}-1}\right)^{-1} < C^\star\right).
\end{equation}
\end{proposition}
Under $H_0$, it can be shown that
$-2\ln \Lambda$ has  an asymptotic chi-square distribution with one
degree of freedom when $n$, sample size, goes to infinity, thus $C^\star\approx \exp\left\{-\frac{1}{2}\chi_{1,1-\gamma}\right\}$,
where $\chi_{v,p}$ is the $p$-th quantile of a chi-square
distribution with $v$ degrees of freedom.
\section{Illustrative examples}
\subsection*{Example 1}
Table \ref{telephone:data} shows the times between 48 (in
minutes) consecutive telephone calls to a company's switchboard,
as presented by Castillo {\em et. al.} (2005).
\begin{table}
\caption{Times (in minutes) between 48 consecutive calls.}
\label{telephone:data}       
\begin{tabular}{cccccccc}
\hline
1.34&0.14&0.33&1.68&1.86&1.31&0.83&0.33\\
2.20&0.62&3.20&1.38&0.96&0.28&0.44&0.59\\
0.25&0.51&1.61&1.85&0.47&0.41&1.46&0.09\\
2.18&0.07&0.02&0.64&0.28&0.68&1.07&3.25\\
0.59&2.39&0.27&0.34&2.18&0.41&1.08&0.57\\
0.35&0.69&0.25&0.57&1.90&0.56&0.09&0.28\\
\hline
\end{tabular}
\centering
\end{table}
Assuming that the times between the consecutive telephone calls
follow the exponential distribution $Exp(\sigma)$,
Castillo {et. al.} (2005) obtained the MLE of $\sigma$ based on
the complete data as $\hat{\sigma}_C=0.934$.  The corresponding
record data, obtained from these complete data, are presented in
Table \ref{castillo:data:record}.
\begin{table}
\caption{Record data arising from times (in minutes) between 48
consecutive calls.}
\label{castillo:data:record}       
\begin{tabular}{cccccc}
\hline $i$& 1&2&3&4&5\\
$R_i$&1.34&0.14&0.09&0.07&0.02\\
$K_i$&1&22&2&1&22 \\
\hline
\end{tabular}
\centering
\end{table}
By assuming $Exp(\sigma)$-model, the MLE of $\sigma$ on the basis of
record data is obtained to be $\hat{\sigma}_{0}=1.022$ while by assuming $W(\alpha,\sigma)$-model, from \eqref{weibull:log:likelihood:derivative} and \eqref{weibull:log:likelihood:derivative:zero:alpha},  MLEs of $\alpha$ and $\sigma$ is obtained as $\hat{\alpha}=1.1815$ and $\hat{\sigma}=0.8181$,
respectively. To calculate the GOF statistics, Table \ref{gof:statistics:castillo:data:record} is useful.
\begin{table}
\caption{GOF from times between 48 consecutive calls.}
\label{gof:statistics:castillo:data:record}       
\begin{tabular}{cccccclc}
\hline $i$& $r_i$ & $k_i$&$r_{(i)}$& $k_{(i)}$& $\hat{\Phi}_i$ & $\bar{F}_n(r_{(i)})=\hat{\Phi}_1\cdots\hat{\Phi}_i$ &$\hat{\bar{F}}_0(r_{(i)})=\exp\{-(r_{(i)}/\hat{\sigma})^{\hat{\alpha}}\}$\\\hline
1&1.34&1 &0.02&22&$\frac{48-1}{48}$&$\frac{47}{48}=0.9792$&0.9876\\
2&0.14&22&0.07&1 &$\frac{26-1}{26}$&$\frac{47}{48}\times \frac{25}{26}=0.9415$&0.9467\\
3&0.09&2 &0.09&2 &$\frac{25-1}{25}$&$\frac{47}{48}\times \frac{25}{26}\times \frac{24}{25}=0.9038$&0.9290\\
4&0.07&1 &0.14&22&$\frac{23-1}{23}$&$\frac{47}{48}\times \frac{25}{26}\times \frac{24}{25}\times\frac{22}{23}=0.8646$&0.8832\\
5&0.02&22&1.34&1 &$\frac{1-1}{1}$&0&0.1667\\
\hline
\end{tabular}
\centering
\end{table}
From Table \ref{gof:statistics:castillo:data:record}, we conclude that
\[D_n=0.6979,\ \ W_n^2=,5.5140\ \ DS_n=8.8604\]
Letting $\gamma=0.05$, from Table \ref{tab:gof:weibull}, three
approaches lead to accept Weibull model for this data. For
testing exponential model against the alternative Weibull model,
GLR statistics is obtained as
\[\Lambda=\left(\frac{\sum_{i=1}^{m}k_i
r_i^{\hat{\alpha}}}{\sum_{i=1}^{m}k_i
r_i}\right)^m\left(\hat{\alpha}^m\left\{\prod_{i=1}^m r_i
\right\}^{\hat{\alpha}-1}\right)^{-1}=1.4765,\]
or, $-2\ln\Lambda=0.3896630654$ which gives the $p-value=0.5324766591$. This supports exponential assumption  by Castillo {et. al.} (2005).
A graph of likelihood function is given in Figure \ref{fig:castillo:likelihood}.
\begin{figure}
\centering
\includegraphics[angle=0,scale=0.40]{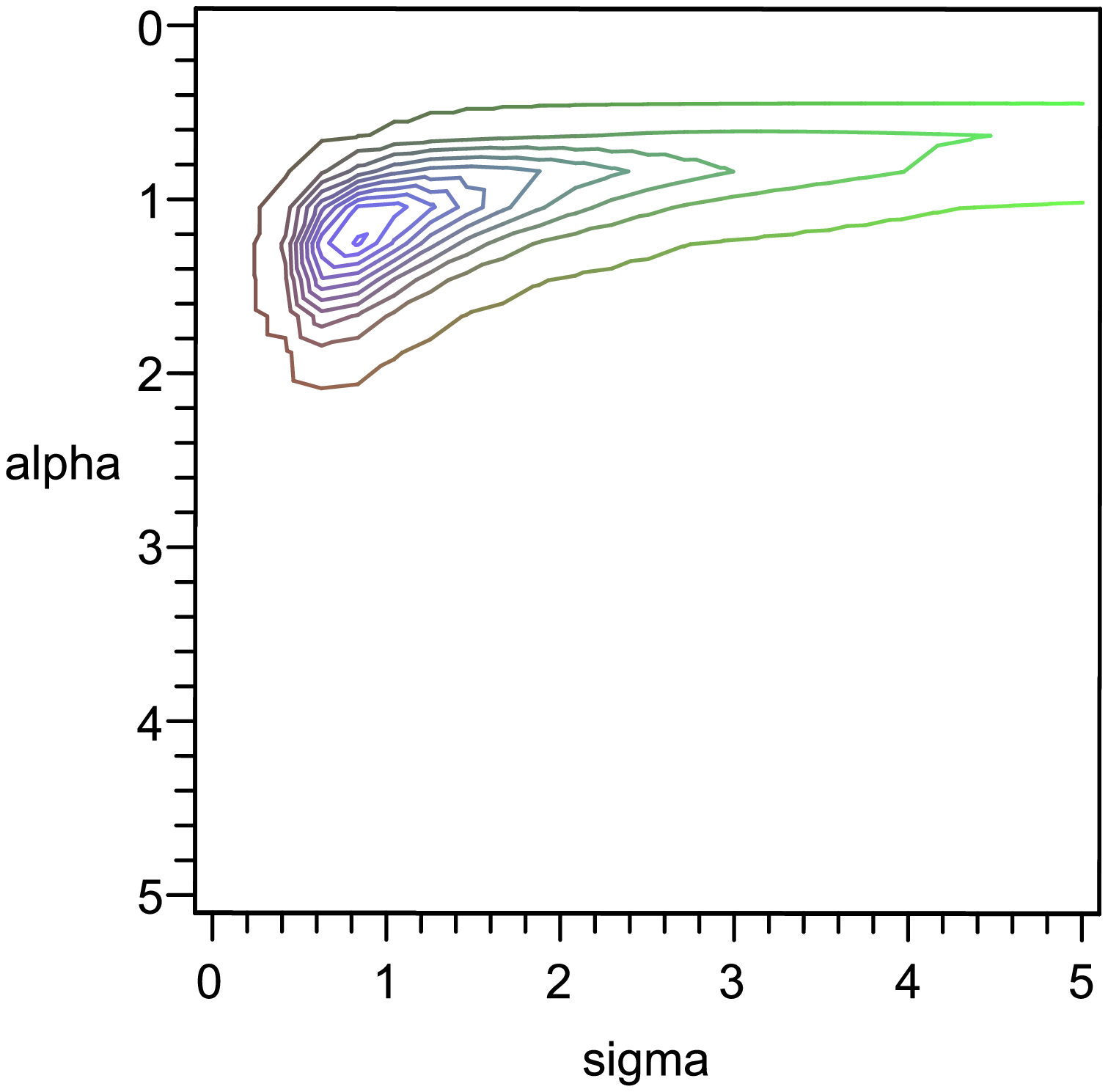}
\caption{Contour  plot of likelihood function \eqref{like:weibull} on the basis of data in Table \ref{castillo:data:record}.}\label{fig:castillo:likelihood}
\end{figure}
\subsection*{Example 2}
Samaniego and Whitaker (1986) presented record data arising from successive failure times of air conditioning units in Boeing aircraft on plan 7914 consists of $n=24$ failure times. The data is given in Table \ref{samaniego:1986:data:record}.
\begin{table}
\caption{Successive minima, plane 7914.}\label{samaniego:1986:data:record}
\begin{tabular}{ccccc}
\hline $i$& 1&2&3&4\\
$R_i$&50&44&22&3\\
$K_i$&1&3&2&18\\
\hline
\end{tabular}
\centering
\end{table}
They approximated these data by $Exp(\sigma)$-model and estimated the mean life $\sigma$ as $\hat{\sigma}_{0}=70$. Under $W(\alpha,\sigma)$-model, the MLEs of $\alpha$ and $\sigma$ are obtained as
\[\hat{\alpha}=1.598743046,\ \ \ \hat{\sigma}=51.42746441,\]
respectively.
Therefore, $-2\ln\Lambda=1.580279376$ which gives the $p-value=0.2087204561$. This supports exponential assumption  by Samaniego and Whitaker (1986).
A graph of likelihood function is given in Figure \ref{fig:samaniego:1986:likelihood}.
\begin{figure}
\centering
\includegraphics[angle=0,scale=0.5]{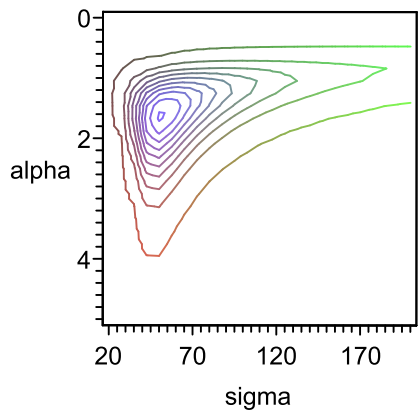}
\caption{Contour  plot of likelihood function \eqref{like:weibull} on the basis of data in Table \ref{samaniego:1986:data:record}}\label{fig:samaniego:1986:likelihood}
\end{figure}

\subsection*{Example 3}
Samaniego and Whitaker (1988) simulated a random sample with size $n=30$ from $W(\alpha=4,\sigma=1)$-model and  record data arising from this sample is presented in Table \ref{samaniego:1988:data:record}.
\begin{table}
\caption{Simulated record data from $W(\alpha=4,\sigma=1)$.}\label{samaniego:1988:data:record}
\begin{tabular}{ccccc}
\hline $i$& 1&2&3&4\\
$R_i$&0.879&0.765&0.735&0.220\\
$K_i$&3&2&2&23\\
\hline
\end{tabular}
\centering
\end{table}
Assuming $Exp(\sigma)$-model, MLE of the mean life $\sigma$ is $\hat{\sigma}_{0}=2.67425000$. By assuming $W(\alpha,\sigma)$-model, the MLEs of $\alpha$ and $\sigma$ are obtained as
\[\hat{\alpha}=3.316071956,\ \ \ \hat{\sigma}=0.9728468503,\]
respectively.
Therefore, $-2\ln\Lambda=7.911804336$ which gives the $p-value=0.0049113232$. This supports departure from exponential assumption. A graph of likelihood function is given in Figure \ref{fig:samaniego:1988:likelihood}.
\begin{figure}
\centering
\includegraphics[angle=0,scale=0.5]{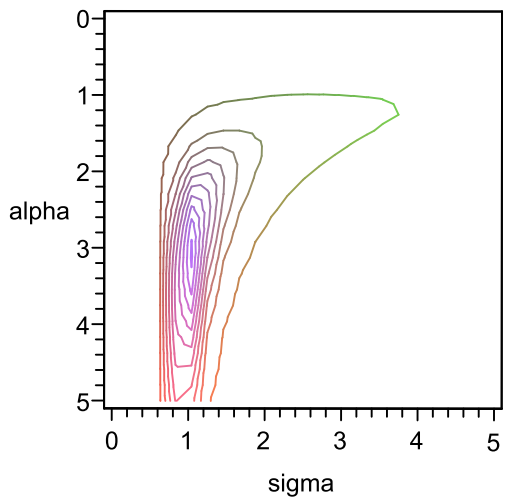}
\caption{Contour  plot of likelihood function \eqref{like:weibull} on the basis of data in Table \ref{samaniego:1988:data:record}}\label{fig:samaniego:1988:likelihood}
\end{figure}

\section{Concluding Remarks}
In this paper, Kolmogorov-Smirnov and  Cramer-von Misses type goodness of fit tests as well as a new weighted statistics for record data were proposed. These statistics were used to goodness of fit test for Weibull model. 
We suggest the following discipline to analyze record data:
First step is to test weibull model using the proposed GOF tests in Section \ref{gof:weibull}. Were it accepted, GLR test in Section \ref{gof:expo} for the exponentially model. Use the statistical procedures for record data arising from exponential model provided that the exponential model were accepted. See Samaniego and Whitaker (1986), Arnold {\em et. al.} (1998), Doostparast (2009), Doostparast and Balakrishnan (2010). If the exponentially was rejected, one can use the results of Hoinkes and Padgett (1994). If the weibull model was rejected, one can use the non-parametric results of Samaniego and Whitaker (1988).

Following Samaniego and Whitaker (1988), one can consider the problem when the available data are arising from $L$ sequence of random variables. More precisely, assume that $L$ independent samples
\[Y_{i1},Y_{i2},\cdots,Y_{i,n_i},\ \ \ 1\leq i\leq L,\]
each of size $n_i$, are obtained sequentially from $F$. The resulting records are $R_{i1},K_{i1},\cdots$, $R_{im_i},K_{im_i}$ for $i=1,2,\cdots,L$ where $K_{i m_i}=n_i-\sum_{j=1}^{m_i-1}K_{ij}$.
Similarly, the NPMLE of the survival function at point $t$ is obtained as
\begin{equation}\label{npmle:record:m:sample}
\hat{\bar{F}}(t)=\prod_{i:r_{(i)}\leq t}\frac{\sum_{j=i}^{m^\star}
k_{(j)}-1}{\sum_{j=i}^{m^\star} k_{(j)}},
\end{equation}
where $m^\star=\sum_{i=1}^L m_i$, $\{r_{(i)},i=1,2,\cdots,m^\star\}$ be the order observed record values in the $L$ samples combined and  $\{k_{(i)},i=1,2,\cdots,m^\star\}$ the induced order statistics for the associated $k_{ij}$.
To carry out the impact of $L$ on the power of the GOF tests, one can conduct a simulation study. 


\end{document}